\documentclass[oneside,reqno]{amsart}

\usepackage{arxiv_preamble}

\title{The \texorpdfstring{$18\cdot 2^t+1$}{18*2^t+1} Triangle-Maximal Series of Straight Lines}
\author{Roman Parpalak}
\author{Denis Utkin}
\thanks{Email: \texttt{roman@parpalak.com} (R.\,Parpalak),
  \texttt{ud1@mail.ru} (D.\,Utkin).}
\date{April 23, 2026}

\begin{document}
\begin{abstract}
Given $n$ lines in general position in the plane, how many
bounded triangular faces can the arrangement have?
We construct a straight-line affine arrangement of $19$ lines
satisfying the conditions of the iterative construction by
Bartholdi, Blanc, and Loisel, thereby obtaining an infinite series
of straight-line arrangements attaining the maximum number of
bounded triangles for every $n=18\cdot 2^t+1$.
The conditions are verified by computer-assisted interval
and combinatorial checks.
A computational search over $n=21$, $23$, $27$ lines provides strong
evidence against the existence of further base configurations
compatible with the known iterative constructions, but reveals
arrangements allowing a single iterative step that yield arrangements
of $41$ and $45$ lines with $533$ and $645$ bounded triangles,
respectively, each matching the upper bound.
\end{abstract}

\maketitle

\section{Introduction}

The Kobon triangle problem~\cite{oeis} asks for the maximum possible number of bounded triangular
faces determined by $n$ lines in the real affine plane. A classical
counting argument---each of the $n$ lines has at most $n-2$ bounded segments,
and each triangle consumes three of them---gives the upper bound
$\lfloor n(n-2)/3 \rfloor$.
For small~$n$, arrangements matching this bound can often be found by hand or
by computer search. For large~$n$, the construction of F{\"u}redi and
Pal{\'a}sti~\cite{furedi1984} asymptotically approaches the upper bound, and for
certain infinite but sparse series of~$n$ the exact maximum is known.
Constructions that produce such families are rare and especially valuable.

The general idea behind infinite series is iterative doubling: one
starts with a \emph{base configuration} of $n$ lines that is optimal and
satisfies certain geometric conditions, then inserts $n-1$ additional lines
near a distinguished line or at infinity to obtain an arrangement of $2n-1$ lines
that is again optimal and satisfies analogous conditions,
so the process can be repeated ad infinitum.

Two constructions of this type have been described in the literature,
due to Forge and Ram{\'i}rez Alfons{\'i}n~\cite{forge1998} and to
Bartholdi, Blanc, and Loisel~\cite{bartholdi2007}. They both yield several
affine series, including straight-line series for
\begin{equation}\label{eq:known-series}
  n = 2\cdot 2^t + 1,
  \qquad
  n = 6\cdot 2^t + 1,
  \qquad
  n = 14\cdot 2^t + 1,
  \qquad t\ge 0,
\end{equation}
and a series of affine pseudoline arrangements for
\begin{equation*}
  n = 18\cdot 2^t + 1.
\end{equation*}
A pseudoline arrangement is an arrangement of curves, not necessarily
straight, each pair crossing exactly once.
This last family has $n\equiv 1\pmod 6$, and
\cite[Theorem~1.3]{bartholdi2007} proves that the exact affine upper
bound
\begin{equation*}
  \left\lfloor \frac{n(n-2)}{3}\right\rfloor
  =
  \frac{n(n-2)-2}{3}
\end{equation*}
is attained by the pseudoline arrangements of this series.

It has remained open whether this last series can be realized by straight lines.
The main contribution of this paper is the discovery of a simple affine
arrangement of $19$ straight lines satisfying the
hypotheses of~\cite[Proposition~3.1]{bartholdi2007}. This arrangement serves as
the straight-line base configuration for the iterative doubling in the case
$n=18\cdot 2^t+1$. Since the base configuration achieves the optimal triangle
count for $19$ lines, the iteration applies directly and yields the optimal
straight-line family for all $n=18\cdot 2^t+1$.

Finding the base configuration requires a systematic search. The number of
combinatorially distinct optimal arrangements grows rapidly with~$n$, while the
conditions imposed by the iterative constructions
of~\cite{forge1998,bartholdi2007} are quite restrictive: a candidate
must simultaneously be stretchable and have a specific crossing pattern
along the distinguished line or a specific pattern of slopes. We therefore organized the
computation into several stages: enumeration of optimal pseudoline arrangements,
classification up to projective equivalence, testing stretchability via
numerical optimization, and checking compatibility with the iterative constructions.

We also applied this search procedure to $n=21$, $23$, $25$, $27$ ($n=25$ restricted to
arrangements containing a bounded pentagonal face none of whose edges belongs
to a bounded triangle). No arrangement compatible with the iterative constructions
of~\cite{forge1998,bartholdi2007} was found, providing strong computational evidence
that no further infinite straight-line families exist for
these values. For $n=21$, $23$, and $25$,
however, we found arrangements allowing a single iterative step,
yielding optimal affine arrangements for $n=41$ and $n=45$
and an affine arrangement of $49$ lines that, together with the line at infinity,
attains the projective upper bound (Appendix~\ref{sec:appendix-one-step}).

\begin{figure}[!ht]
\centering
\input{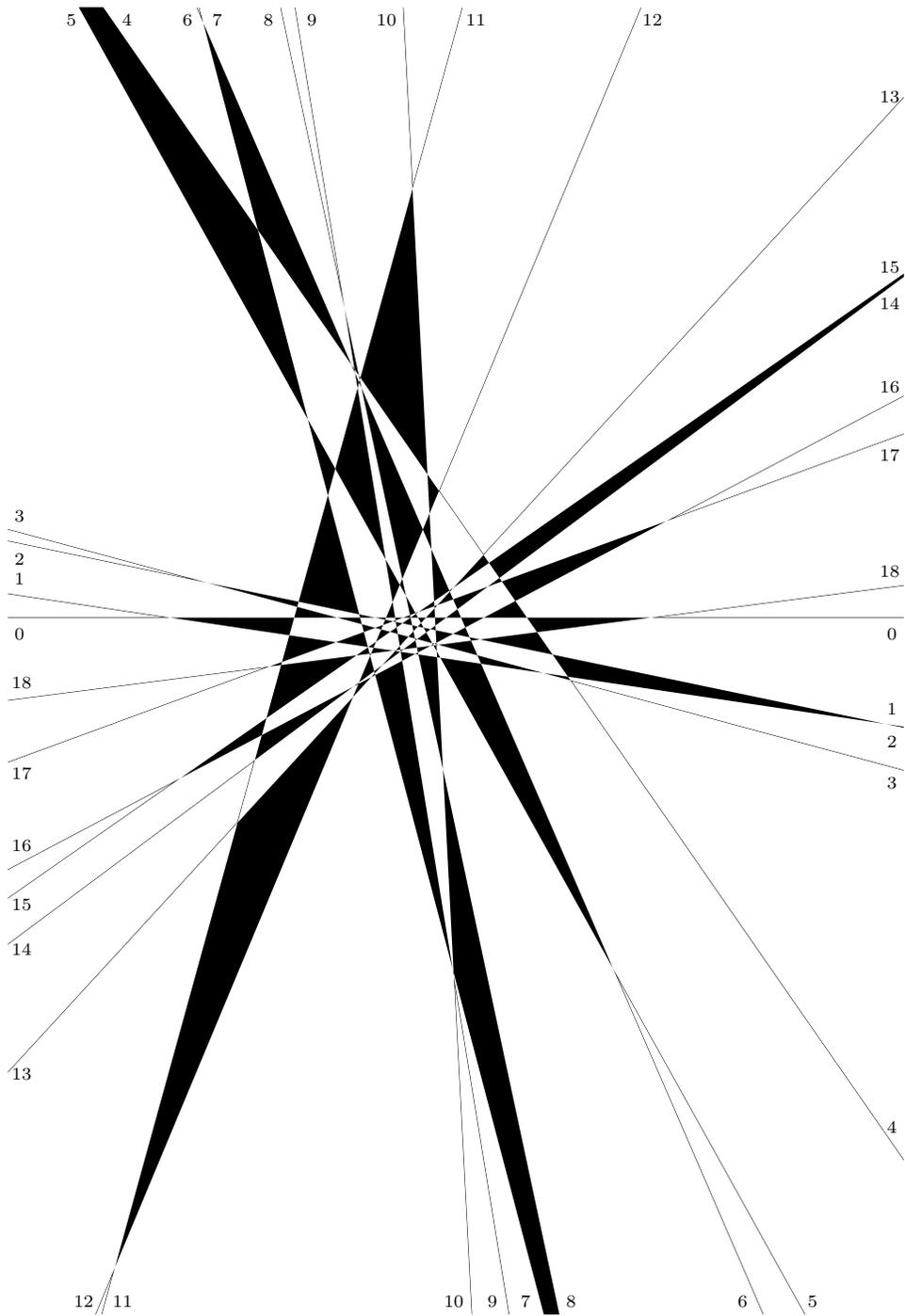}
\caption{The base configuration for $n=19$ straight lines.
}
\label{fig:base-configuration}
\end{figure}

The structure of the paper is as follows. In
Section~\ref{sec:input-statement} we recall the hypotheses of the iterative
construction from~\cite{bartholdi2007}. Section~\ref{sec:combinatorics} defines
the $\mathrm{O}$-matrix encoding of the combinatorial type.
The straight-line base configuration is
established in Section~\ref{sec:main-result}, and
Section~\ref{sec:series} applies the iterative construction to derive the
infinite straight-line family. Section~\ref{sec:search} describes the computer
search that produced the base configuration, and
Section~\ref{sec:discussion} discusses the search for base
configurations at other values of~$n$.
Appendix~\ref{sec:certified-package} describes the computer-assisted
verification, and Appendix~\ref{sec:appendix-one-step} records one-step
iterative examples at $n=41$, $45$, and $49$.

\section{Definitions and Setup for the Iterative Construction}\label{sec:input-statement}

An arrangement of lines is called \emph{simple} if no point belongs to more
than two of its lines. We work with simple affine arrangements in
$\mathbb{R}^2$ and write $a_3(\cA)$ for the number of bounded triangular
faces of an arrangement~$\cA$. A distinguished line $Y_0$ \emph{touches} a
triangle if it contains one of its sides.

Throughout the paper we use the best affine upper bounds for odd $n$
from \cite[Theorem~1]{blanc2011}: for every simple affine arrangement
of $n$ lines or pseudolines,
\begin{equation}\label{eq:upper-bound}
  a_3(\cA)\le
  \begin{cases}
    (n(n-2)-2)/3 & \text{if } n\equiv 1\pmod{6}, \\
    n(n-2)/3 & \text{if } n\equiv 3,5\pmod{6}.
  \end{cases}
\end{equation}

We use the following affine formulation of the iterative step from
\cite[Proposition~3.1 and Remark~3.2]{bartholdi2007}.

\begin{proposition}[Affine iterative step]\label{prop:iterative-step}
Let $n\ge 2$ be an even number, and let
\[
  \cA=\{Y_0,L_1,\dots,L_n\}
\]
be a simple affine arrangement of $n+1$ lines, given by the equations
\[
  Y_0:\ y=0,
  \qquad
  L_i:\ y=m_i(x-a_i).
\]
Assume that
\[
  \{a_1,\dots,a_{n-2}\}=\{\pm\tan(k\pi/n)\mid k=1,\dots,n/2-1\},
\]
and
\[
  -\frac{1}{n}<a_{n-1}<0<a_n<\frac{1}{n},
\]
and that $Y_0$ touches exactly $n-1$ triangles of the arrangement $\cA$.
Then there exist $n$ lines $M_1,\dots,M_n$ such that
the affine arrangement
\[
  \cB=\{Y_0,L_1,\dots,L_n,M_1,\dots,M_n\}
\]
of $2n+1$ lines is simple and has exactly $n^2$ more
triangles than $\cA$, and $Y_0$ touches exactly $2n-1$ triangles of $\cB$.
\end{proposition}

\begin{remark}[Iteration condition]\label{rem:iteration-condition}
If moreover $|a_{n-1}|$ and $|a_n|$ are smaller than $1/(2n)$, then the new
arrangement $\cB$ again satisfies the hypotheses of the proposition. Hence the
construction can be iterated provided these two special parameters are chosen
arbitrarily small.
\end{remark}

For the base configuration studied in this paper
(see Figure~\ref{fig:base-configuration}), the relevant value is $n=18$,
so it has $19$ lines in total. In Section~\ref{sec:main-result},
we prove the existence of a straight-line affine arrangement
$
  \{Y_0,L_1,\dots,L_{18}\}
$
with
\[
  \{a_1,\dots,a_{18}\}=\{\pm\tan(k\pi/18)\mid k=1,\dots,8\} \cup \{-\eps,+\eps\},
\]
where $\eps>0$ is small enough and the hypotheses of the proposition and
remark above hold.

\section{Combinatorial Encoding by the \texorpdfstring{$\mathrm{O}$}{O}-Matrix}\label{sec:combinatorics}

For a simple affine arrangement of $n$ lines, choose a generic directed line
$\gamma$ (the \emph{sweep line}) that is not parallel to any line of the
arrangement and whose initial position is sufficiently far from all
intersection points. Label the lines $L_0,L_1,\dots,L_{n-1}$
(or simply $0, 1, \dots, n-1$) in the order in which
$\gamma$ meets them at its initial position. As $\gamma$ sweeps
through the plane, each intersection point $L_i\cap L_j$ is encountered
exactly once; append~$j$ to the row of~$L_i$ and~$i$ to the row of~$L_j$
at that moment. The result is an $n\times(n-1)$ matrix
which we call the $\mathrm{O}$-matrix (order matrix) of the arrangement.

The numbering $L_0,L_1,\dots,L_{n-1}$ induced by $\gamma$ coincides with the
increasing angular order of the lines; the same numbering is used in the
parameter tables and figures throughout the paper.

This construction applies to pseudoline arrangements as well; in that case
the $\mathrm{O}$-matrix encodes the same data as a rank-$3$ oriented matroid
(equivalently, an allowable sequence of permutations).
We use the $\mathrm{O}$-matrix since its row-based format is convenient: each row
directly records the sequence of crossings along a line, while adjacency of
labels within rows detects bounded faces
(Figure~\ref{fig:omatrix-example}).

\begin{figure}[ht]
\centering
\definecolor{triEdgeA}{RGB}{220,38,38}
\definecolor{triEdgeB}{RGB}{56,163,28}
\definecolor{triEdgeC}{RGB}{136,48,212}
\newcommand{\edgeacolor}[1]{\textcolor{triEdgeA}{\mathbf{#1}}}
\newcommand{\edgebcolor}[1]{\textcolor{triEdgeB}{\mathbf{#1}}}
\newcommand{\edgeccolor}[1]{\textcolor{triEdgeC}{\mathbf{#1}}}
\begin{minipage}[t]{0.32\linewidth}

\begin{tikzpicture}[scale=1.2,baseline=2.9cm]
  \coordinate (L1a) at (-0.8,2.5);
  \coordinate (L1b) at (1.6,0);
  \coordinate (L2a) at (0.1,2.5);
  \coordinate (L2b) at (1.4,0);
  \coordinate (L3a) at (0.35,2.5);
  \coordinate (L3b) at (-0.45,0);
  \coordinate (L4a) at (1.15,2.5);
  \coordinate (L4b) at (-0.6,0);
  \coordinate (P12) at (intersection of L1a--L1b and L2a--L2b);
  \coordinate (P13) at (intersection of L1a--L1b and L3a--L3b);
  \coordinate (P14) at (intersection of L1a--L1b and L4a--L4b);
  \coordinate (P23) at (intersection of L2a--L2b and L3a--L3b);
  \coordinate (P24) at (intersection of L2a--L2b and L4a--L4b);
  \coordinate (P34) at (intersection of L3a--L3b and L4a--L4b);
  \fill[gray!20] (P13) -- (P14) -- (P34) -- cycle;
  \draw (L1a) -- (L1b) node[pos=0.04,left,font=\small]{$0$};
  \draw (L2a) -- (L2b) node[pos=0.04,left,font=\small]{$1$};
  \draw (L3a) -- (L3b) node[pos=0.04,right,font=\small]{$2$};
  \draw (L4a) -- (L4b) node[pos=0.04,right,font=\small]{$3$};
  \draw[triEdgeA, line width=1.6pt] (P13) -- (P14);
  \draw[triEdgeB, line width=1.6pt] (P13) -- (P34);
  \draw[triEdgeC, line width=1.6pt] (P14) -- (P34);
  \foreach \p in {P12,P13,P14,P23,P24,P34}
    \fill (\p) circle (0.8pt);
  \draw[dashed] (-1.0,2.65) -- ++(2.5,0) node[right]{$\gamma$}
    ++(0.2,-0.2) edge[solid, ->] ++(0,-0.5);
\end{tikzpicture}
\end{minipage}
\hfill
\begin{minipage}[t]{0.32\linewidth}
$\mathrm{O}$-matrix

\medskip
\renewcommand{\arraystretch}{1.2}
$\begin{array}{c|ccc}
  & 1\text{st} & 2\text{nd} & 3\text{rd} \\
\hline
  L_0 & \edgeacolor{2} & \edgeacolor{3} & 1 \\
  L_1 & 2 & 3 & 0 \\
  L_2 & 1 & \edgebcolor{0} & \edgebcolor{3} \\
  L_3 & 1 & \edgeccolor{0} & \edgeccolor{2}
\end{array}$

\end{minipage}
\hfill
\begin{minipage}[t]{0.32\linewidth}
\raggedright
\setlength{\baselineskip}{1.22\baselineskip}
The shaded triangle $(L_0,L_2,L_3)$ corresponds to
the adjacent pairs:\\
$\edgeacolor{2},\edgeacolor{3}$ in row~$L_0$;\;\\
$\edgebcolor{0},\edgebcolor{3}$ in row~$L_2$;\;\\
$\edgeccolor{0},\edgeccolor{2}$ in row~$L_3$.
\end{minipage}
\caption{A simple arrangement of four lines, its $\mathrm{O}$-matrix, and the
correspondence between faces and adjacency in the matrix.}
\label{fig:omatrix-example}
\end{figure}
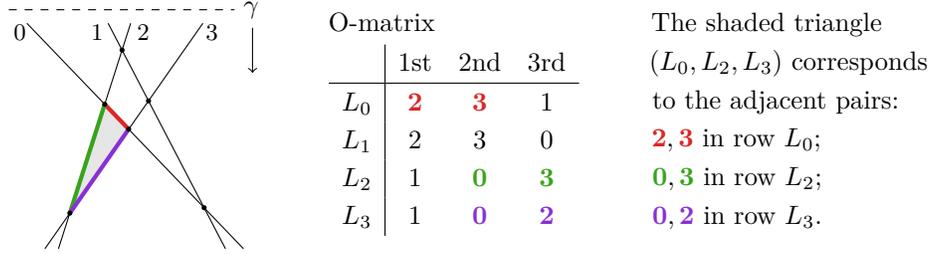

\begin{lemma}[$\mathrm{O}$-matrix criterion for bounded faces]\label{lem:face-criterion}
Let $(L_{i_1},\dots,L_{i_k})$ be a cyclically ordered list of distinct
pseudolines in a simple arrangement. Then
these pseudolines bound a bounded $k$-gonal face if and only if for every
$s=1,\dots,k$ (with $i_0=i_k$ and $i_{k+1}=i_1$)
the two labels $i_{s-1}$ and $i_{s+1}$ appear in adjacent positions in the row
of $L_{i_s}$ of the $\mathrm{O}$-matrix:
$$
\mathrm{O}[i_s] = (\dots, i_{s\pm1}, i_{s\mp1}, \dots),
\qquad s=1,\dots,k.
$$
\end{lemma}

\begin{proof}
Along $L_{i_s}$, the segment between $L_{i_s}\cap L_{i_{s-1}}$ and
$L_{i_s}\cap L_{i_{s+1}}$ is an edge of a bounded face precisely when it
contains no other crossing, i.e., $i_{s-1}$ and $i_{s+1}$ are adjacent in the
row of $L_{i_s}$. Imposing this for every $s$ gives a closed polygonal
boundary.
\end{proof}

\section{The Base Configuration}\label{sec:main-result}

In this section we present the base configuration of $19$ straight lines
(Figure~\ref{fig:base-configuration}) and establish its properties that allow us
to apply the iterative step from Proposition~\ref{prop:iterative-step}.

\begin{figure}[ht]
\centering
\resizebox{1.0\textwidth}{!}{\input{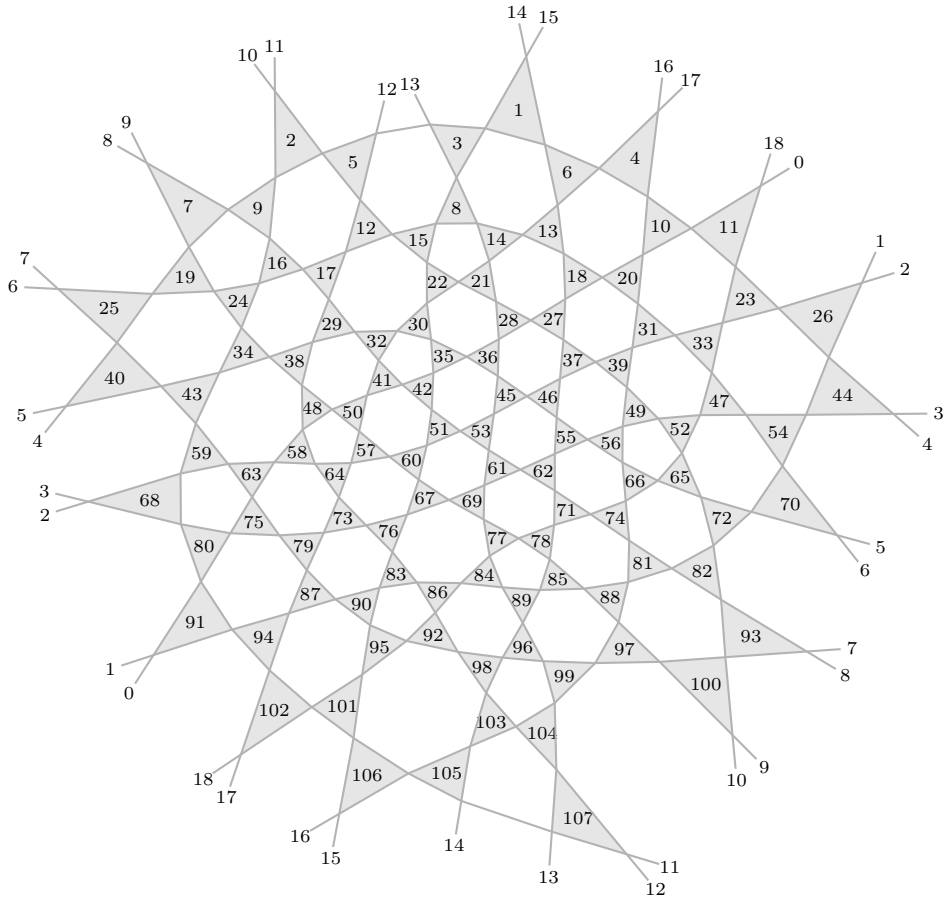}}
\caption{Pseudoline diagram of the base configuration with all $107$
bounded triangles numbered. Lines are labeled $0,1,\dots,18$.}
\label{fig:pseudolines}
\end{figure}

The base configuration is the normalized family of straight lines
\begin{equation}\label{eq:seed19-normalized-family}
  Y_0:\ y=0,
  \qquad
  L_i:\ y=m_i(x-a_i),
  \qquad i=1,\dots,18,
\end{equation}
where $a_8=-\eps$, $a_{15}=+\eps$, and the remaining $a_i$ values are
$\pm\tan(k\pi/18)$, $k=1,\dots,8$. The slopes $m_i$ and the parameter $\eps$
may vary in the following range:
\providecommand{\SeedNineteenEpsLower}{10^{-12}}
\providecommand{\SeedNineteenEpsUpper}{10^{-4}}
\providecommand{\SeedNineteenDeltaM}{10^{-4}}
\refstepcounter{equation}\label{eq:seed19-parameter-box}
\[
  \eps \in \left(0,\SeedNineteenEpsUpper\right),
  \qquad
  m_i \in \left[m_i^\ast-\delta_m,m_i^\ast+\delta_m\right],
  \qquad
  \delta_m = \SeedNineteenDeltaM.
  \tag{\theequation}
\]
The exact values of slope centers $m_i^\ast$ and $x$-intercepts $a_i$ are:
\begingroup
\[
\renewcommand{\arraystretch}{1.1}
\begin{array}{ll@{\qquad}ll}
 m_{1}^\ast = -0.18006, & a_{1} = -\tan(8\pi/18) & m_{10}^\ast = -23.07892, & a_{10} = +\tan(3\pi/18) \\
 m_{2}^\ast = -0.25231, & a_{2} = -\tan(4\pi/18) & m_{11}^\ast = 4.39005, & a_{11} = -\tan(7\pi/18) \\
 m_{3}^\ast = -0.32584, & a_{3} = -\tan(6\pi/18) & m_{12}^\ast = 2.90204, & a_{12} = -\tan(3\pi/18) \\
 m_{4}^\ast = -1.74521, & a_{4} = +\tan(7\pi/18) & m_{13}^\ast = 1.31845, & a_{13} = +\tan(2\pi/18) \\
 m_{5}^\ast = -2.17941, & a_{5} = +\tan(1\pi/18) & m_{14}^\ast = 0.90192, & a_{14} = +\tan(4\pi/18) \\
 m_{6}^\ast = -2.79557, & a_{6} = +\tan(5\pi/18) & m_{15}^\ast = 0.84501, & a_{15} = +\eps \\
 m_{7}^\ast = -4.59586, & a_{7} = -\tan(5\pi/18) & m_{16}^\ast = 0.64068, & a_{16} = +\tan(6\pi/18) \\
 m_{8}^\ast = -5.67203, & a_{8} = -\eps & m_{17}^\ast = 0.44393, & a_{17} = -\tan(1\pi/18) \\
 m_{9}^\ast = -7.39905, & a_{9} = -\tan(2\pi/18) & m_{18}^\ast = 0.15548, & a_{18} = +\tan(8\pi/18) \\
\end{array}
\]
\endgroup

The corresponding $\mathrm{O}$-matrix is listed in \eqref{eq:seed19-omatrix},
and Figure~\ref{fig:pseudolines} shows the associated pseudoline diagram.

\refstepcounter{equation}
\begingroup
\noindent\hfill
\begin{minipage}[c]{0.72\linewidth}
\footnotesize
\begin{verbatim}
[
  [1, 11, 3, 7, 2, 12, 9, 17, 8, 15, 5, 13, 10, 14, 6, 16, 4, 18],
  [0, 11, 17, 7, 15, 12, 18, 13, 14, 9, 16, 8, 10, 5, 6, 3, 4, 2],
  [3, 11, 7, 0, 12, 17, 9, 15, 8, 13, 5, 14, 10, 16, 6, 18, 4, 1],
  [2, 11, 0, 7, 17, 12, 15, 9, 13, 8, 14, 5, 16, 10, 18, 6, 1, 4],
  [5, 7, 6, 9, 8, 11, 10, 12, 13, 15, 14, 17, 16, 0, 18, 2, 1, 3],
  [4, 7, 11, 9, 12, 8, 17, 15, 0, 13, 2, 14, 3, 16, 18, 10, 1, 6],
  [7, 4, 9, 11, 8, 12, 10, 15, 13, 17, 14, 0, 16, 2, 18, 3, 1, 5],
  [6, 4, 5, 11, 2, 0, 3, 17, 1, 15, 18, 12, 14, 13, 16, 9, 10, 8],
  [9, 4, 11, 6, 12, 5, 17, 0, 15, 2, 13, 3, 14, 18, 16, 1, 10, 7],
  [8, 4, 6, 11, 5, 12, 0, 17, 2, 15, 3, 13, 18, 14, 1, 16, 7, 10],
  [11, 4, 12, 6, 15, 17, 13, 0, 14, 2, 16, 3, 18, 5, 1, 8, 7, 9],
  [10, 4, 8, 6, 9, 5, 7, 2, 3, 0, 1, 17, 18, 15, 16, 14, 13, 12],
  [4, 10, 6, 8, 5, 9, 0, 2, 17, 3, 15, 1, 18, 7, 14, 16, 13, 11],
  [4, 15, 6, 17, 10, 0, 5, 2, 8, 3, 9, 18, 1, 14, 7, 16, 12, 11],
  [15, 4, 17, 6, 0, 10, 2, 5, 3, 8, 18, 9, 1, 13, 7, 12, 16, 11],
  [14, 4, 13, 6, 10, 17, 5, 0, 8, 2, 9, 3, 12, 1, 7, 18, 11, 16],
  [17, 4, 0, 6, 2, 10, 3, 5, 18, 8, 1, 9, 7, 13, 12, 14, 11, 15],
  [16, 4, 14, 6, 13, 10, 15, 5, 8, 0, 9, 2, 12, 3, 7, 1, 11, 18],
  [0, 4, 2, 6, 3, 10, 5, 16, 8, 14, 9, 13, 1, 12, 7, 15, 11, 17]
]
\end{verbatim}

\end{minipage}
\hfill(\theequation)\label{eq:seed19-omatrix}
\endgroup

\begin{theorem}[Straight-line base configuration at $n=19$]\label{thm:main-seed}
For every choice of parameters in \eqref{eq:seed19-parameter-box},
the normalized family
\eqref{eq:seed19-normalized-family} is a simple affine arrangement of $19$
straight lines satisfying:
\begin{enumerate}[label=(\arabic*)]
\item $Y_0$ is the line $y=0$;
\item each $L_i$ has the form $y=m_i(x-a_i)$;
\item
$
  \{a_1,\dots,a_{18}\}
  =
  \{\pm\tan(k\pi/18)\mid k=1,\dots,8\}\cup\{-\eps,+\eps\}
$, with $\eps<1/36$;
\item $Y_0$ touches exactly $17$ triangles;
\item the arrangement has exactly $107$ bounded triangular faces.
\end{enumerate}
\end{theorem}

The resulting family satisfies the hypotheses of
Proposition~\ref{prop:iterative-step} (up to relabeling), with the special
pair $(L_8,L_{15})$ given by $a_8=-\eps$ and $a_{15}=+\eps$.
Since the iterative construction requires $\eps$ arbitrarily close to $0$,
but our computer-assisted verification (recorded in
Appendix~\ref{sec:certified-package}) uses interval arithmetic on closed
bounded intervals, we split the argument
at an auxiliary cutoff $\SeedNineteenEpsLower$:
Lemma~\ref{lem:interval-realization} covers
$[\SeedNineteenEpsLower,\SeedNineteenEpsUpper]$ by direct computation, and
Lemma~\ref{lem:bbl-match} extends the result to
$(0,\SeedNineteenEpsLower)$ by an affine-linearity argument.

\begin{lemma}[Interval realization of the $\mathrm{O}$-matrix]\label{lem:interval-realization}
For every $\eps\in[\SeedNineteenEpsLower,\SeedNineteenEpsUpper]$ and every
choice of slopes in \eqref{eq:seed19-parameter-box},
the normalized family \eqref{eq:seed19-normalized-family} is a simple affine
arrangement realizing exactly the $\mathrm{O}$-matrix
\eqref{eq:seed19-omatrix}.

In other words, throughout this parameter range, the crossing order along each
line is fixed by the $\mathrm{O}$-matrix, no two lines are parallel,
and no triple intersections occur.
\end{lemma}

\begin{proof}
The combinatorial order along $Y_0$ depends on the parameters
$a_i$, which are determined by the construction. It can be checked directly that
this order is compatible with the first row $\mathrm{O}[0]$ of~\eqref{eq:seed19-omatrix},
i.e., $a_{\mathrm{O}[0,j]}$ for $j=1,\dots,18$ forms an increasing sequence.
For every other line $L_i$, the slope $m_i$ is nonzero on the whole
parameter range \eqref{eq:seed19-parameter-box}, so the order of intersection points
along $L_i$ coincides with the order of their $y$-coordinates.

The intersection point of $L_i$ and $L_j$ has $y$-coordinate
\begin{equation}
  y_{i,j}=\frac{m_i m_j(a_i-a_j)}{m_i-m_j},\label{eq:y-coordinate}
\end{equation}
so each row of the $\mathrm{O}$-matrix is equivalent to a list of strict
inequalities comparing these $y$-coordinates. For instance, in row $4$ the label $5$
precedes the label $7$, hence one of the required inequalities is
\[
  \frac{m_4 m_5(a_4-a_5)}{m_4-m_5}
  >
  \frac{m_4 m_7(a_4-a_7)}{m_4-m_7}.
\]
In this way the $\mathrm{O}$-matrix \eqref{eq:seed19-omatrix} can be transformed into a finite system of
strict row-order inequalities
\begin{equation}\label{eq:row-order-ineqs}
  y_{i,\,\mathrm{O}[i,j]} - y_{i,\,\mathrm{O}[i,j+1]} > 0,
  \qquad i=1,\dots,18,\  \  j=1,\dots,17,
\end{equation}
together with the slope-order conditions
\begin{equation}\label{eq:slope-separation}
  \frac{1}{m_1} < \frac{1}{m_2} < \cdots < \frac{1}{m_{18}},
\end{equation}
which order the lines by direction. (Since $1/m_i=\cot\theta_i$, where
$\theta_i\in(0,\pi)$ is the angle that $L_i$ makes with the positive $x$-axis,
and since $\cot$ is monotone on $(0,\pi)$, this chain of conditions
is equivalent to a strict ordering of the angles $\theta_i$.) In particular,
\eqref{eq:slope-separation} guarantees the absence of parallel lines.
Together, \eqref{eq:row-order-ineqs} and \eqref{eq:slope-separation} yield
more than $300$ rational inequalities, which we verify
using the computer-assisted interval-arithmetic calculations recorded in
Appendix~\ref{sec:certified-package} (though in principle this could be done by hand).
It follows that the crossing order along every line is fixed
throughout the range and that no two lines are parallel.

Strict row order also excludes triple intersections: if
$L_i\cap L_j = L_i\cap L_k$, then $y_{i,j}=y_{i,k}$, contradicting
the strict inequality. Together with
\eqref{eq:slope-separation}, this proves simplicity and the exact realization of the
$\mathrm{O}$-matrix \eqref{eq:seed19-omatrix}.
\end{proof}

\begin{lemma}[Triangle counts from the $\mathrm{O}$-matrix]\label{lem:omatrix-consequences}
Every simple affine arrangement realizing the $\mathrm{O}$-matrix
\eqref{eq:seed19-omatrix} has
\begin{enumerate}[label=(\arabic*)]
\item exactly $107$ bounded triangular faces;
\item $Y_0$ touches exactly $17$ triangles.
\end{enumerate}
\end{lemma}

\begin{proof}
Both counts can be verified directly by inspecting the pseudoline diagram in
Figure~\ref{fig:pseudolines} and checking that the $\mathrm{O}$-matrix
\eqref{eq:seed19-omatrix} encodes the same combinatorial type. Alternatively,
by Lemma~\ref{lem:face-criterion}, bounded faces and their sizes are determined
by adjacency of labels in the rows of the $\mathrm{O}$-matrix; the
computer-assisted verification (Appendix~\ref{sec:certified-package}) carries out
this check and confirms exactly $107$ bounded triangles, of which $17$
touch~$Y_0$.
\end{proof}

\begin{lemma}[Small-parameter extension of the realization]\label{lem:bbl-match}
For every $\eps\in(0,\SeedNineteenEpsUpper)$ and every choice of slopes in the
intervals \eqref{eq:seed19-parameter-box}, the normalized family
\eqref{eq:seed19-normalized-family} is a simple
affine arrangement realizing the $\mathrm{O}$-matrix
\eqref{eq:seed19-omatrix}. In particular,
$\eps<\SeedNineteenEpsUpper<1/36=1/(2\cdot 18)$, so the
small-parameter condition of Remark~\ref{rem:iteration-condition} is satisfied.
\end{lemma}

\begin{proof}
Lemma~\ref{lem:interval-realization} covers
$\eps\in[\SeedNineteenEpsLower,\SeedNineteenEpsUpper]$. It remains to extend
the realization to $(0,\SeedNineteenEpsLower)$. The crossing order along $Y_0$
and the slope-order conditions \eqref{eq:slope-separation} depend only on
the $a_i$ ordering and the slope intervals, so they are unaffected by
shrinking~$\eps$.

Fix any choice of slopes $m_i$ in the intervals
\eqref{eq:seed19-parameter-box}. Since only $a_8=-\eps$ and $a_{15}=+\eps$
depend on~$\eps$, each row-order difference $y_{i,j}-y_{i,k}$
in~\eqref{eq:row-order-ineqs} is affine-linear in~$\eps$. By the
computer-assisted verification (Appendix~\ref{sec:certified-package}),
at $\eps=0$ and for every choice of slopes in the full slope
box~\eqref{eq:seed19-parameter-box}, exactly two such differences
vanish---corresponding to the triangle touching $Y_0$ that collapses at the
origin---and none becomes negative. Lemma~\ref{lem:interval-realization},
which also holds over the full slope box, shows that all differences are
positive at $\eps=\SeedNineteenEpsLower$. An affine-linear function that is
nonnegative at $0$ and positive at $\SeedNineteenEpsLower$ is positive on the
entire interval $(0,\SeedNineteenEpsLower]$; since both bounds are uniform
over the slope box, the conclusion holds for every admissible choice of
slopes. Combined with the interval verification, the realization extends to
all $\eps\in(0,\SeedNineteenEpsUpper)$.
\end{proof}

\begin{proof}[Proof of Theorem~\ref{thm:main-seed}]
The key intermediate step is to show that the normalized family
\eqref{eq:seed19-normalized-family} realizes the
$\mathrm{O}$-matrix \eqref{eq:seed19-omatrix}, from which the triangle counts
follow.
Lemma~\ref{lem:interval-realization} establishes this realization for every
$\eps\in[\SeedNineteenEpsLower,\SeedNineteenEpsUpper]$ and every choice of
slopes in \eqref{eq:seed19-parameter-box}.
Lemma~\ref{lem:bbl-match} extends it to all
$\eps\in(0,\SeedNineteenEpsUpper)$ by an affine-linearity argument.
Lemma~\ref{lem:omatrix-consequences} then derives the touching condition and the
triangle count from the $\mathrm{O}$-matrix. Together, these three lemmas
establish all parts of the theorem for the full parameter
range~\eqref{eq:seed19-parameter-box}.
\end{proof}

\section{The Optimal Series \texorpdfstring{$18\cdot 2^t+1$}{18·2ᵗ+1}}\label{sec:series}

Since the base configuration of Theorem~\ref{thm:main-seed} already satisfies
the hypotheses of the affine iterative step
(Proposition~\ref{prop:iterative-step}), the construction of
\cite{bartholdi2007} applies directly and yields the following.

\begin{theorem}[Straight-line optimal series for $18\cdot 2^t+1$ lines]\label{thm:straight-series}
For every integer $t\ge 0$, there exists a simple affine arrangement of
$n=18\cdot 2^t+1$ straight lines attaining the affine upper
bound~\eqref{eq:upper-bound}:
\[
  a_3
  =\left\lfloor \frac{n(n-2)}{3}\right\rfloor
  =108\cdot 4^t-1.
\]
\end{theorem}

\begin{proof}
By Theorem~\ref{thm:main-seed}, the base configuration $\cA_0$ is a simple
affine arrangement of $19$ straight lines with $a_3(\cA_0)=107$ bounded
triangles and $Y_0$ touching exactly $17$ of them. By
Lemma~\ref{lem:bbl-match}, the realization holds for all
$\eps\in(0,\SeedNineteenEpsUpper)$; in particular, $\eps$ can be chosen
arbitrarily small, so the iteration condition of
Remark~\ref{rem:iteration-condition} is satisfied and
Proposition~\ref{prop:iterative-step} can be applied repeatedly. Writing
$n_t=18\cdot 2^t$, each step produces a simple affine arrangement $\cA_{t+1}$
of $2n_t+1$ straight lines with $n_t^2$ more bounded triangles than $\cA_t$.
Therefore
\[
  a_3(\cA_t)
  =107+\sum_{s=0}^{t-1}(18\cdot 2^s)^2
  =107+324\cdot\frac{4^t-1}{3}
  =108\cdot 4^t-1
  =\frac{(n_t+1)(n_t-1)-2}{3},
\]
which is the upper bound~\eqref{eq:upper-bound} for $n=n_t+1$.
\end{proof}

\section{How the Base Configuration Was Found}\label{sec:search}

This section describes the computer search that produced the base
configuration of Theorem~\ref{thm:main-seed}. The proof does not depend on
this material; we include it to document the methodology.
For the computer-assisted checks used in the proof,
see Appendix~\ref{sec:certified-package}.

The search requires two stages.
First, we enumerated pseudoline arrangements by generating
allowable sequences of permutations realizing the target triangle count
given by the upper bound~\eqref{eq:upper-bound}.
Second, we extracted the corresponding $\mathrm{O}$-matrices and attempted to
find straight-line realizations compatible with the constraints imposed by the
iterative constructions of~\cite{forge1998,bartholdi2007}.
The enumeration at $n=19$ was complete, but completeness is not
required: a single suitable candidate suffices for the proof of
Theorem~\ref{thm:main-seed}.

In general, stretching a pseudoline arrangement means realizing the
crossing orders encoded by its $\mathrm{O}$-matrix. If, in a fixed row, the
label $j$ must precede the label $k$, then the $y$-coordinates~\eqref{eq:y-coordinate}
of the corresponding intersection points
must satisfy an inequality of the form
\begin{equation*}
\frac{a_i-a_j}{s_i-s_j} < \frac{a_i-a_k}{s_i-s_k},
\end{equation*}
or, equivalently, after clearing denominators and taking signs into account,
in a form more suitable for computation,
\begin{equation}
(s_i-s_k)(a_i-a_j) \lessgtr (s_i-s_j)(a_i-a_k),\label{eq:general-order-ineq}
\end{equation}
where $s_i=1/m_i$. Thus the combinatorial realization problem becomes a finite
feasibility problem for row-order inequalities together with the
slope-order conditions~\eqref{eq:slope-separation}.

The present search of base configurations for the iterative construction
in Proposition~\ref{prop:iterative-step} is more restrictive.
It also fixes the $x$-intercept parameters $a_i$, while the reciprocal slope variables
$s_i=1/m_i$ remain unknown. With this choice, every adjacent row-order
comparison becomes linear in the variables $s_i$. The stretching step therefore
reduces to finding a feasible point of the system of linear inequalities
\begin{equation}
\ell_r(s_1,\dots,s_{n-1})\le -\eta,\qquad r=1,\dots,R,\label{eq:linear-order-ineq}
\end{equation}
where the linear forms $\ell_r$ are obtained from~\eqref{eq:general-order-ineq} and
$\eta>0$ is a fixed safety margin. We solve this feasibility problem with the
HiGHS linear optimization package~\cite{huangfu2018dual}, which provides
reliable infeasibility certificates.

This search generated candidate arrangements but was not part of the proof.
After obtaining a floating-point realization, the
normalization script described in Appendix~\ref{sec:certified-package}
determined exact rational parameters and the admissible intervals recorded in
Theorem~\ref{thm:main-seed}.

\section{Discussion}\label{sec:discussion}

The main result of this paper is the discovery of a base configuration for
$n=19$ that yields an infinite straight-line affine series attaining the optimal
number of bounded triangles for every $n=18\cdot 2^t+1$.
The base configuration was found by a systematic computational search with
two goals: to identify base configurations for infinite iteration, and to
detect arrangements that, while not supporting repeated iteration, may still
admit a single realizable iterative step.
The search was restricted to odd values of~$n$, since the iterative
constructions of~\cite{forge1998,bartholdi2007} apply only to odd~$n$;
enumeration of optimal pseudoline arrangements is also simpler in this case.

\textbf{Search for infinite series.}
The table below summarizes the results,
obtained in three stages:
enumeration of projective equivalence classes,
numerical stretchability testing, and
checking compatibility with iterative conditions.

\medskip
\noindent\begin{tabular}{|l|r|r|r|r|r|r|r|}
 \hline
 $n$                  & 15 & 17 & 19         & 21       & 23  & 25, 5-gon defect & 27 \\
 \hline
 Projective classes   & 1  & 3  & 1312       & 18       & 112 & 2324             & 56646 \\
 \hline
 Stretchable classes  & 1  & 3  & $\geq1306$ & $\geq11$ & $\geq51$ & $\geq122$   & $\geq348$   \\
 \hline
 Base configurations  & 1  & 3  & 170        & 0        & 0   & 0                & 0 \\
 \hline
\end{tabular}
\medskip

\emph{Projective classes.}
After enumerating optimal pseudoline arrangements, we identified projectively
equivalent arrangements and counted them only once.

\emph{Stretchable classes.}
Next, for each of the projective classes, we tried to realize the
corresponding pseudoline arrangements by straight lines numerically. We
minimized by gradient descent the sum of squared violations of the row-order
inequalities~\eqref{eq:general-order-ineq}. Thus, the numbers in this row are
only lower bounds: a
negative result may simply mean that the computation did not find a straight-line representative.

\emph{Base configurations.}
Finally, we tested all enumerated optimal arrangements reconstructed
from projective classes for compatibility with the conditions required
for the iterative constructions of~\cite{forge1998,bartholdi2007}.
This check is reduced to the feasibility of the linear
system~\eqref{eq:linear-order-ineq}, so the absence of a solution
can be guaranteed, unlike with the gradient-based stretching procedure.
None of the tested arrangements for $n=21$, $23$, $27$
yielded a base configuration, so these zero entries
should be viewed as stronger negative evidence.
The columns $n=15$ and $n=17$ reproduce known small cases
from the series~\eqref{eq:known-series}.
The case $n=25$ is discussed separately below.

Thus, for $n=21$, $23$, $27$, the
computations provide strong evidence against the existence of new
infinite iterative series of type~\cite{forge1998,bartholdi2007}.

\textbf{Infinite series for $\boldsymbol{n=25}$.}
This case for affine arrangements is already covered by the known
series~\eqref{eq:known-series}, so we focus on the projective extension.
A natural way to produce projective arrangements is to add the line at
infinity. Unbounded faces become bounded, and the number of new triangles
depends on $n\bmod 6$.

For $n\equiv 1\pmod{6}$, the upper bound~\eqref{eq:upper-bound} implies that at least two
bounded segments do not belong to any bounded triangle.
In the affine plane, these \emph{unused segments}
belong to unbounded faces of optimal arrangements (see
$L_{13}L_{11}L_{14}$ and $L_{12}L_{4}L_{13}$ in
Figure~\ref{fig:pseudolines}), so the affine
upper bound is still attained. In the projective plane all faces are bounded,
and the unused segments form faces whose every edge is unused;
we call such faces \emph{defects}. In all optimal
arrangements that we enumerated for $n\equiv 1\pmod{6}$, the defect
is either a single pentagon or a pair of quadrilaterals.

After projective extension, the upper bound for $n$ lines with
$n\equiv 2\pmod{6}$
\begin{equation}\label{eq:blanc-projective-bound}
  p_3(\cA)\le \frac{n(n-1)-5}{3}
\end{equation}
is attained by pseudoline arrangements only in the presence of a
pentagonal defect, and to our knowledge no corresponding straight-line
configuration has been reported~\cite{blanc2011}.
Accordingly, the column $n=25$ in the
table concerns only pentagonal-defect arrangements, since this is the
only subclass potentially yielding a new projective-maximal family.
Within this subclass, none of the enumerated arrangements passes the
iterative compatibility test.

By contrast, for $n\equiv 3,5\pmod{6}$ every bounded segment of an
optimal arrangement is used (no defects arise), and adding the line at
infinity simply produces $n$ new triangles. The additional combinatorial freedom introduced by
defects explains why the counts at $n=19$ are much larger than at
$n=21$ and $n=23$.

\textbf{One-step realizable examples.}
Compatibility with the iterative conditions is a stronger requirement than
ordinary stretchability, since it is designed to support repeated application
of iterative constructions. Its failure therefore rules out an infinite
iterative family of this type, but does not exclude the possibility that one
or more finite iterative steps still produce straight-line realizable
arrangements. Explicit realizable arrangements of
$41$ and $45$ lines, obtained from arrangements of $21$ and $23$ lines after a
single iterative step, are recorded in Appendix~\ref{sec:appendix-one-step}.
Besides these examples, we found a realizable arrangement
of $49$ lines obtained from a $25$-line arrangement with a pentagonal defect
after a single iterative step. After adding the line at infinity, this yields
a projective arrangement with $50$ lines and $815$ triangles, attaining the
projective upper bound \eqref{eq:blanc-projective-bound}.

Finally, we observe that as $n$ increases, the proportion of stretchable
arrangements among the enumerated optimal pseudoline arrangements appears to
decrease. We conjecture that the number of infinite series of
type~\cite{forge1998,bartholdi2007} arising from stretchable
optimal arrangements is finite.

\appendix

\section{Computer-Assisted Verification}\label{sec:certified-package}

The computer-assisted part of the proof consists of two independent
verification steps, each producing a machine-readable certificate.

\textbf{The first step} concerns the geometric realization. Starting from the
line data of a candidate arrangement and the target $\mathrm{O}$-matrix~\eqref{eq:seed19-omatrix},
the normalization script reconstructs the normalized affine
arrangement~\eqref{eq:seed19-normalized-family}
\[
  Y_0:\ y=0,
  \qquad
  L_i:\ y=m_i(x-a_i),
\]
recovers the exact values of the $a_i$ parameters
($\pm\tan(k\pi/18)$, $\pm\eps$), and verifies that the row-order
inequalities~\eqref{eq:row-order-ineqs} and slope-order
conditions~\eqref{eq:slope-separation} hold throughout the
parameter box \eqref{eq:seed19-parameter-box}. The resulting certificate
records:
\begin{itemize}[label=$\triangleright$]
\item the choice of the distinguished line $Y_0$ and the exact labels of the
  $a_i$ parameters;
\item an interval for the special parameter $\eps$;
\item strict ordering of intersections along each line, with a positive
  safety margin;
\item positive slope-separation for every pair of slopes;
\item the direction ordering~\eqref{eq:slope-separation} of the lines,
  verified over the full slope range;
\item the $\eps=0$ checks: substituting $\eps=0$ and evaluating every
  row-order difference using interval arithmetic over the full slope
  range~\eqref{eq:seed19-parameter-box}, exactly two differences vanish
  (the ordering of $Y_0$ and $L_{15}$ along $L_8$, and of $Y_0$ and $L_8$
  along $L_{15}$; all three lines meet at the origin when $a_8=a_{15}=0$)
  and no difference becomes negative;
\item the verdict that the normalized arrangement realizes the target
  $\mathrm{O}$-matrix.
\end{itemize}

\textbf{The second step} concerns the combinatorial properties of the base
configuration. Using the face criterion of Lemma~\ref{lem:face-criterion},
the structural verifier enumerates bounded faces directly from adjacency
relations in the rows of the $\mathrm{O}$-matrix and confirms:
\begin{itemize}[label=$\triangleright$]
\item $Y_0$ touches exactly $17$ triangles;
\item the arrangement has exactly $107$ bounded triangular faces.
\end{itemize}

This separation of steps is not essential. In principle, the combinatorial structure and
triangle counts could be recovered directly from the line equations in a single
verification step. We instead use the $\mathrm{O}$-matrix as an intermediate
object fixing the combinatorial type, so that the two certificates mirror the
two parts of the proof: realization of the $\mathrm{O}$-matrix by the line
arrangement, and verification of the required properties of the
$\mathrm{O}$-matrix itself.

All input data, verification scripts, output certificates,
and detailed format specifications are available at our GitHub repository
\cite{repo}.

\section{One-Step Iterative Examples at \texorpdfstring{$n=41,45,49$}{n=41,45,49}}\label{sec:appendix-one-step}

This appendix records concrete examples of arrangements that
do not satisfy the conditions required for repeated
iteration in the sense of~\cite{forge1998,bartholdi2007}, but for which a
single iterative step may still produce a straight-line realizable arrangement.

Figures~\ref{fig:41-from-21-lines},~\ref{fig:45-from-23-lines},
and~\ref{fig:49-from-25-lines} show the arrangements, and
Tables~\ref{tab:41-coefficients},~\ref{tab:45-coefficients},
and~\ref{tab:49-coefficients} record the line coefficients in the affine form
$y = mx + b$.
The first is a realizable $41$-line arrangement
obtained from a $21$-line arrangement,
with $533$ bounded triangles.
The second is the analogous result for $45$ lines
obtained from a $23$-line arrangement, with $645$ bounded triangles.
The third is a realizable $49$-line arrangement obtained from a $25$-line
arrangement with a pentagonal defect; after adding the line at infinity, it
yields a projective arrangement of $50$ lines attaining the upper
bound~\eqref{eq:blanc-projective-bound}.

The examples use the construction of~\cite{forge1998} rather
than~\cite{bartholdi2007} solely because the former produces more visually
appealing figures. These two constructions are related by a projective
transformation, so the construction of~\cite{bartholdi2007} could be applied
equally well.

The line coefficients listed in the tables can be pasted into an online
viewer that allows zooming in on the arrangements and inspecting polygon
statistics. The viewer, along with a gallery of further examples, is
available at~\cite{repo}.

\begin{figure}[!hp]
\centering
\input{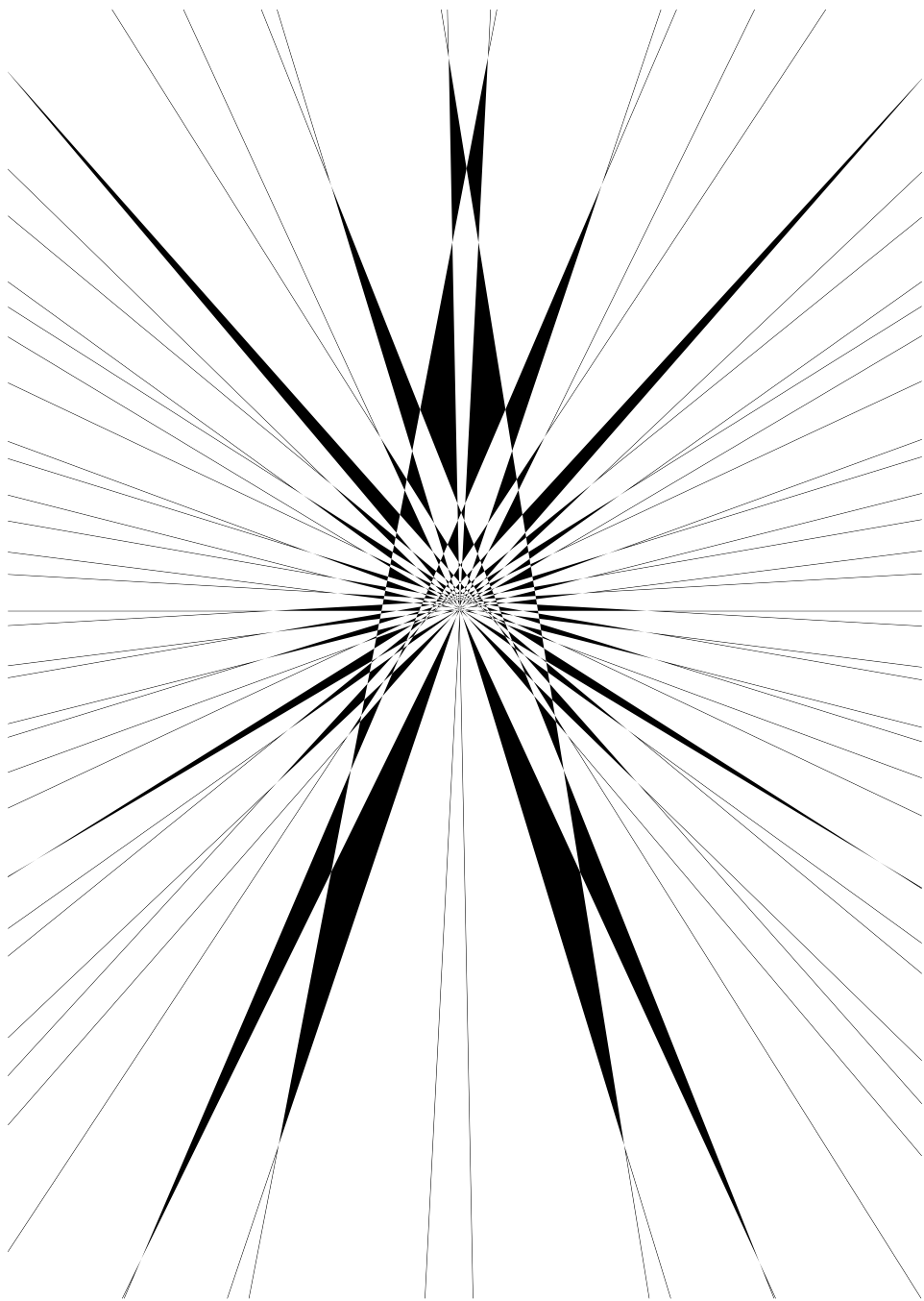}
\caption{A realizable $41$-line arrangement with $533$ bounded triangles,
matching the upper bound $\lfloor n(n-2)/3 \rfloor$, obtained
from a $21$-line arrangement after one iterative step of~\cite{forge1998}.}
\label{fig:41-from-21-lines}
\end{figure}

\begin{figure}[!hp]
\centering
\captionof{table}{Line coefficients for the $41$-line arrangement of
Figure~\ref{fig:41-from-21-lines}.}\label{tab:41-coefficients}
\small
\begin{verbatim}
m,b
19.6440380158708,-0.578845642445838
5.19983686759458,46.2669483580472
2.97337242178767,-0.255654502748468
2.33409554657142,10.2073308592064
2.03913847183038,-0.634409531311398
1.51939067684971,4.44408039396207
1.14469295714126,-0.383673607397969
1.09198915528216,2.46780787838117
0.947707289018028,-0.581595297292468
0.80896478221994,1.52592194561986
0.708319546739286,-0.46596899057913
0.624766449916582,1.07709697251601
0.592073396328359,-0.522334544621754
0.467419996326499,0.815211228257588
0.364476352936825,-0.358176891237224
0.307097716326852,0.627325466449936
0.252224697244642,-0.606243665819936
0.173239760811204,0.517542416022027
0.125712170009275,-0.500504814655953
0.0567903240132735,0.473964134108348
0,-0.747810827880863
-0.0566407268373554,0.474709003288992
-0.126026378587329,-0.502668706579321
-0.173988094834067,0.520070556147547
-0.25426442214236,-0.611448868231776
-0.309876980168933,0.632838712396089
-0.368480699536813,-0.362861640209878
-0.474087518144302,0.826370333485319
-0.603198691333035,-0.533094500199415
-0.636848012086688,1.09721398633079
-0.724090133958835,-0.477608081683738
-0.829471506779334,1.56371844343208
-0.97587461794496,-0.60078764216201
-1.12998251203546,2.55282240379824
-1.18598949150119,-0.399745439534234
-1.59358094120356,4.65965510727435
-2.17751680508347,-0.680853322380702
-2.51510426865305,10.9989476990021
-3.27532346737191,-0.286935223516981
-6.19714088202275,55.1771409818031
-50.5023827311972,-1.56858159736557
\end{verbatim}
\end{figure}

\begin{figure}[!hp]
\centering
\input{figures/45-from-23-lines.tex}
\caption{A realizable $45$-line arrangement with $645$ bounded triangles,
matching the upper bound $\lfloor n(n-2)/3 \rfloor$, obtained
from a $23$-line arrangement after one iterative step of~\cite{forge1998}.}
\label{fig:45-from-23-lines}
\end{figure}

\begin{figure}[!hp]
\centering
\captionof{table}{Line coefficients for the $45$-line arrangement of
Figure~\ref{fig:45-from-23-lines}.}\label{tab:45-coefficients}
\small
\begin{verbatim}
m,b
1.53430417945565,0.20403146857486
-0.519717793360893,-1.0123770007636
1.22463734160685,0.361302797408153
-0.682592554218052,-1.18681219034012
0.891854978259324,0.316083666124596
-0.870047509959627,-1.47843485760808
0.651476460231005,0.418124154690683
-1.09636719038685,-1.97015219221493
0.396929305083391,0.286795175281954
-1.41989120786806,-2.83983360068664
0.272556432493848,0.499073931073725
-1.87010286558218,-4.37575147679071
0.108540530144518,0.413263071246813
-2.5103136907933,-7.37929773678778
0,0.610545074014735
-3.52443245742084,-15.0932528519642
-0.15596198360546,0.355345732571265
-7.25076505035506,-63.4720982341358
-0.274318853521862,0.389406637042344
-0.393772372893695,0.27252089035818
6.11597424517132,-61.5713675655963
-0.584242811465604,0.433463267312579
2.99158206439091,-15.7352961692028
-0.724111708173206,0.181281981043461
1.96411862253783,-7.32581355453173
-1.08905653302851,0.461023808411459
1.37117853979795,-4.18294030008071
-1.29134149027236,0.418287739507626
0.993450056980779,-2.69385966502562
-1.74197588871076,0.717639658599045
0.746886415976471,-1.9669652534109
-1.89849013127593,0.349476299783515
0.538492493879797,-1.52059345557914
-2.7586087357052,0.658112245930578
0.351699573604613,-1.23335298576353
-3.91743259598345,-0.0661110689538446
0.170375461501831,-1.03320644694924
-47.3722474445319,1.43579772315994
0.0203941950811925,-0.914576078806777
29.6728308554765,1.4775805773327
-0.104237468278218,-0.868327950312221
3.68039533050188,-0.083492590289976
-0.237082763091935,-0.868846495493857
2.41186113356395,0.327174404614809
-0.37333249578809,-0.912399480160171
\end{verbatim}
\end{figure}

\begin{figure}[!hp]
\centering
\input{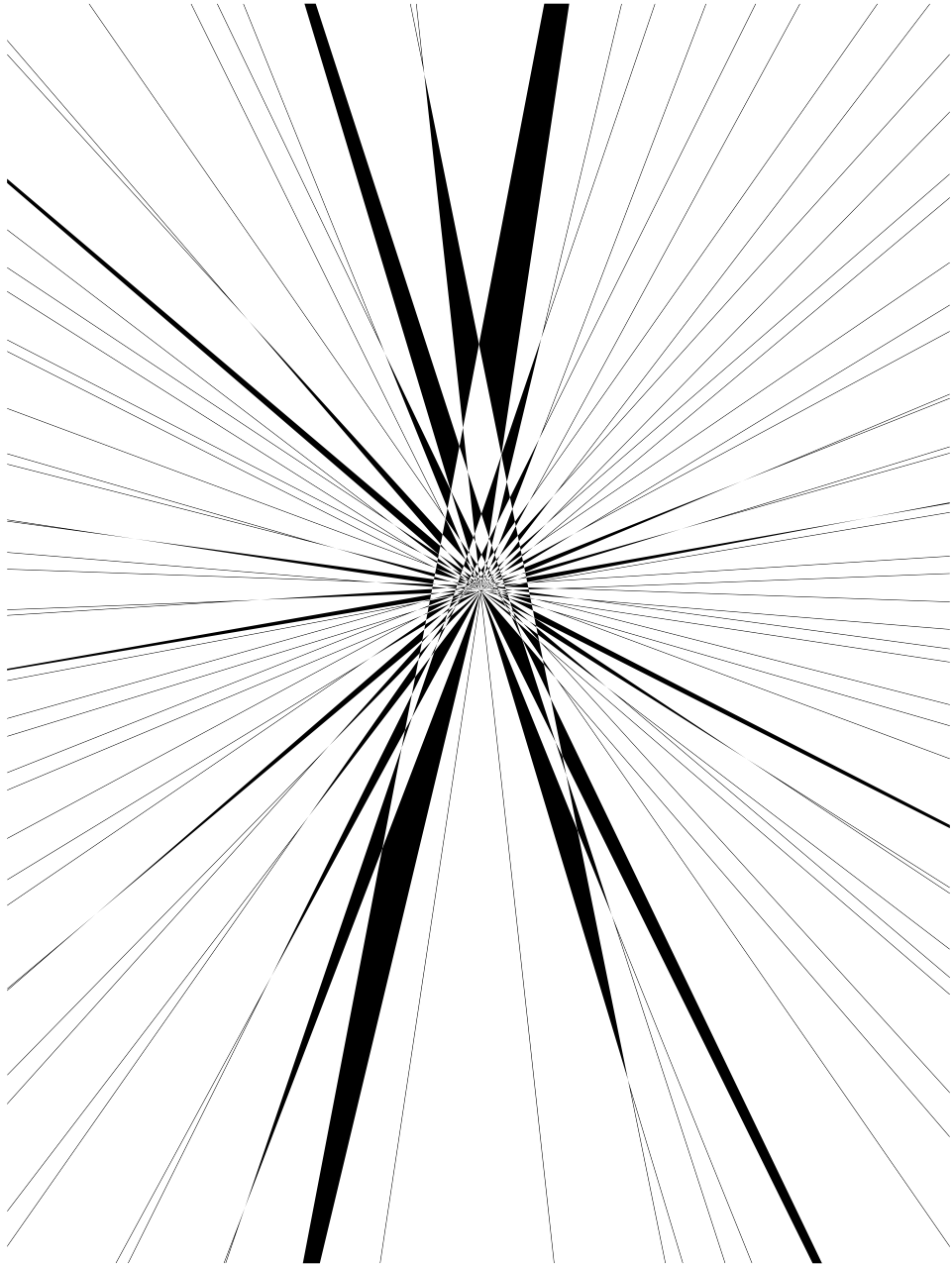}
\caption{A realizable $49$-line arrangement with $766$ bounded triangles, obtained
from a $25$-line arrangement after one iterative step of~\cite{forge1998}.
After adding the line at infinity, one obtains a projective arrangement of
$50$ lines with $815$ triangles, attaining the upper bound
\eqref{eq:blanc-projective-bound}.}
\label{fig:49-from-25-lines}
\end{figure}

\begin{figure}[!hp]
\centering
\captionof{table}{Line coefficients for the $49$-line arrangement of
Figure~\ref{fig:49-from-25-lines}.}\label{tab:49-coefficients}
\small
\begin{verbatim}
m,b
-0.055378708076769124,1.054138063855369
-0.13039037461054212,-1.3959980643410135
-0.21628638153113777,1.0485766871752571
-0.24086463229360836,-1.5686114574697163
-0.40728521659108335,1.104865372689265
-0.46050721241195647,-1.2797443053912798
-0.59296182477965476,1.2708287865048573
-0.80761823680704803,-1.8281090060597618
-0.82074547973660095,1.5959659056220838
-1.0092764439776607,-1.8995108078761407
-1.0593634243657601,2.1374135080387169
-1.2180369677336167,-2.1666035521700011
-1.3526950624000658,2.9342977336359803
-1.3731070875197393,-2.0864498484267684
-1.757707592910895,4.5314353597917485
-1.856336775766769,-2.532829936322365
-2.3040072504210496,7.4246690857667179
-3.2354100193383184,-3.2355509103981555
-3.3216826833550961,15.017640507339127
-3.9514399140639491,-3.8614379126460276
-4.9129277341774316,34.064661882568203
-5.3079993119522983,-4.3878258574257947
-7.879132603103729,109.73470101026696
-14.582890498423433,-10.423512074894735
10.612606135454183,4.9151172487244725
8.326098553648098,128.92689110222864
6.6653341495181158,2.890620690589238
4.6808343059033373,38.32102413232856
4.1688617163513468,1.0427812887831018
3.2019980904994862,17.661308614880117
2.9378704370314499,0.57731234402339238
2.2781860768550528,9.9883094757838897
2.0882829751522372,-0.12126164232259505
1.776996132803764,6.622003669091149
1.6306629541874376,-0.25310827362414134
1.3822593027651811,4.5427756765404954
1.3393818644170152,-0.52991819279353858
1.0884238685766567,3.3338126530893191
0.97302233979402408,-0.63287767614748036
0.85855252347104893,2.5688173628027569
0.67040292676335056,-0.92853611857421647
0.63372188108844307,1.9802278283237849
0.49021153715327626,-0.8845054450885611
0.44929357607535608,1.5962758487632895
0.30115944415503992,-1.1588841863417918
0.27985256875264214,1.3374317593155995
0.26786304416444695,-1.1498380839111284
0.10028277364423278,1.1425695136620051
0.061398603131372152,-1.5052069980467306
\end{verbatim}
\end{figure}

\clearpage
\bibliographystyle{unsrt}
\bibliography{refs}

\begin{thebibliography}{1}

\bibitem{oeis}
OEIS~Foundation Inc.
\newblock Sequence {A006066}: {K}obon triangles.
\newblock \url{https://oeis.org/A006066}, 2026.
\newblock Accessed: April 2026.

\bibitem{furedi1984}
Z.~F{\"u}redi and I.~Pal{\'a}sti.
\newblock Arrangements of lines with a large number of triangles.
\newblock {\em Proceedings of the American Mathematical Society},
  92(4):561--566, 1984.

\bibitem{forge1998}
D.~Forge and J.~L. Ram{\'i}rez~Alfons{\'i}n.
\newblock Straight line arrangements in the real projective plane.
\newblock {\em Discrete \& Computational Geometry}, 20:155--161, 1998.

\bibitem{bartholdi2007}
N.~Bartholdi, J.~Blanc, and S.~Loisel.
\newblock On simple arrangements of lines and pseudo-lines in {$P^2$} and
  {$R^2$} with the maximum number of triangles.
\newblock In {\em Surveys on Discrete and Computational Geometry: Twenty Years
  Later}, volume 453 of {\em Contemporary Mathematics}, pages 105--116.
  American Mathematical Society, 2008.

\bibitem{blanc2011}
J.~Blanc.
\newblock The best polynomial bounds for the number of triangles in a simple
  arrangement of n pseudo-lines.
\newblock {\em Geombinatorics}, 21:5--17, 2011.

\bibitem{huangfu2018dual}
Q.~Huangfu and J.~A.~J. Hall.
\newblock Parallelizing the dual revised simplex method.
\newblock {\em Mathematical Programming Computation}, 10(1):119--142, 2018.

\bibitem{repo}
R.~Parpalak and D.~Utkin.
\newblock Verification data and scripts for the $18\cdot 2^t+1$
  triangle-maximal series.
\newblock \url{https://github.com/parpalak/triangle-maximal-18-series}, 2026.
\newblock Snapshot at commit \texttt{446e782}. Accessed: April 2026.

\end{thebibliography}

\end{document}